\documentclass[10pt,a4paper]{amsart}
\usepackage{amsfonts}
\usepackage{amsthm}
\usepackage{amsmath}
\usepackage{amscd}
\usepackage[latin2]{inputenc}
\usepackage{t1enc}
\usepackage[mathscr]{eucal}
\usepackage{indentfirst}
\usepackage{graphicx}
\usepackage{graphics}
\usepackage{enumitem}
\usepackage{pict2e}
\usepackage{epic}
\usepackage{amssymb}
\usepackage[colorlinks,citecolor=red,urlcolor=blue,bookmarks=false,hypertexnames=true]{hyperref} 
\numberwithin{equation}{section}
\usepackage[margin=3cm]{geometry}

\theoremstyle{plain}
\newtheorem{Th}{Theorem}[section]
\newtheorem{Lemma}[Th]{Lemma}
\newtheorem{Cor}[Th]{Corollary}
\newtheorem{Prop}[Th]{Proposition}
\newtheorem*{Theorem-non}{Theorem}
\newtheorem*{Theorem-non2}{Theorem}
\newtheorem*{Proof-non}{Proof of Theorem \ref{hecketheorem} assuming $\mathbf{Axiom1}$ and $\mathbf{Axiom2}$}
\newtheorem*{Proof-non2}{Proof of $m_{2}$ estimates assuming Proposition 5.1}
\newtheorem*{Proof-non3}{Proof of $m_{1}$ estimates}
 \theoremstyle{definition}

\newtheorem{Rem}[Th]{Remark}
\newtheorem{?}[Th]{Problem}

\renewcommand\rightmark{Ergodic averages with the Hecke eigenvalue square weights and the Piltz divisor function weights.}

\begin{document} 

\title{Ergodic averages with the Hecke eigenvalue square weights and the Piltz divisor function weights}

\author{Jiseong Kim}
\address{The University of Mississippi, Department of Mathematics
Hume Hall 335
Oxford, MS 38677}
\email{Jkim51@olemiss.edu}

\date{}

\begin{abstract} 
In this paper, we prove that the Hecke eigenvalue square for a holomorphic cusp form and the Piltz divisor functions are good weighting functions for the pointwise ergodic theorem.  This partially solves problems suggested by Cuny and Weber.  Additionally, we prove similar results for various other arithmetical functions in the last section.  
\end{abstract}
\maketitle

\section{Introduction}
\noindent Let $(X, \mathcal{A}, \nu, \tau)$ be a measurable dynamical system and let $f \in L^{1}(\nu)$. 
Birkhoff's pointwise ergodic theorem states that 
$$\lim_{n \rightarrow \infty} \frac{1}{n} \sum_{k=0}^{n-1} f\left(\tau^{k}(x) \right)  $$ converges $\nu$-almost everywhere.  

In their paper \cite{CunyWeber}, Cuny and Weber generalized Birkhoff's pointwise ergodic theorem by extending it to encompass various arithmetic weights $g$. They demonstrated that for such weights, 
\begin{equation}\label{ex}\lim_{n \rightarrow \infty} \frac{1}{n} \sum_{k=0}^{n-1}g(k) f\left(\tau^{k}(x) \right)  \end{equation} converges $\nu$-almost everywhere.
 This generalization was motivated by the Weighted Strong Law of Large Numbers, which states that for a sequence of i.i.d., integrable random variables $\{X_{i}\}$ and a non-negative, multiplicative function $h(n)$ with limited fluctuation (for further details, see \cite[Theorem A]{CunyWeber} or \cite{BerkesMullerWeber}), the limit
$$\lim_{n\rightarrow \infty} \frac{1}{\sum_{k=1}^{n} h(k)} \sum_{k=1}^{n} h(k)X_{k} = E(X_{1})$$
holds almost everywhere. The proof in \cite{CunyWeber} for $g=d_{2}(k)$ utilized Bourgain's version of the circle Method. This approach reduced the problem to estimating certain exponential sums. 
In \cite[Section 9]{CunyWeber}, the authors concluded by listing several open problems, one of which was to prove that the Piltz divisor function $d_{v}(k)$ is a suitable weighting function for the pointwise ergodic theorem. Note that, based on similar approaches, Giannitsi \cite{Christina} studied a uniform, scale-free $l^{p}$ estimate and various sparse bounds in the case where $g=d_{2}(k)$.

In this paper, we demonstrate that both the Hecke eigenvalue square $\lambda(k)^{2}$ and $d_{v}(k)$ are indeed suitable weighting functions for the pointwise ergodic theorem. 
\subsection{The Hecke eigenvalue squares}
\noindent Note that the materials provided in this subsection are taken from \cite[Section 14]{IK1}. Let $\mathbb{H}=\{z=x+iy | x \in \mathbb{R}, y \in (0,\infty)\}, G=SL(2,\mathbb{Z}).$ 
Define $j_{\gamma}(z)=(cz+d)$ where 
$\gamma= \left(\begin{matrix}
a & b \\
c & d 
\end{matrix} \right) \in G.$
When a holomorphic function $\Phi: \mathbb{H} \rightarrow \mathbb{C}$ satisfies
\begin{equation}\Phi(\gamma z)= j_{\gamma}(z)^{k}\Phi(z) \;\;\; \textrm{for all }\;\; \gamma \in SL(2,\mathbb{Z}),\nonumber\end{equation}
 it is called a modular form of weight $k.$ 
Let $\Phi(z)$ be a non-constant modular form of weight $k$. Then there exists a Fourier expansion at the cusp $\infty$ given by 
\begin{equation}
    \Phi(z)=\sum_{n=0}^{\infty} b_{\Phi}(n)e(nz)
\end{equation}
where $e(z)=e^{2\pi iz},$ and the normalized Fourier coefficient $a_{\Phi}(n)$ of $\Phi(z)$ is defined by
\begin{equation}
a_{\Phi}(n):=b_{\Phi}(n)n^{-(k-1)/2}.
\end{equation}
 The set of all modular forms of fixed weight $k$ forms a vector space, denoted by $M_k$. We also define the subset $\mathcal{C}_{k} \subset M_k$ consisting of all modular forms in $M_k$ which have a zero constant term.
For any positive integer $n$, the $n$th Hecke operator on $\mathcal{C}_{k}$ is defined as follows:
$$ (\mathcal{T}_{n}\Phi)(z):= \frac{1}{n} \sum_{ad=n} a^{k} \sum_{b(\rm{mod} \thinspace \it{d})} \Phi(\frac{az+b}{d})$$
for all $\Phi \in \mathcal{C}_{k}.$
\noindent There exists an orthogonal basis of $\mathcal{C}_{k}$ consisting of eigenfunctions for all Hecke operators $\mathcal{T}_n$, which are referred to as Hecke cusp forms. 
If $\Phi$ is a Hecke cusp form, then the eigenvalues $\lambda(n)$ of the $n$th Hecke operator satisfy 
$$a_{\Phi}(n)=a_{\Phi}(1)\lambda(n)$$ 
and 
$$\lambda(m)\lambda(n)= \sum_{d|(m,n)} \lambda \left(\frac{mn}{d^{2}}\right).$$
We now present one of our main results, which is Theorem \ref{hecketheorem}, the pointwise ergodic theorem with weight $\lambda(k)^2$.
\begin{Th}\label{hecketheorem}$[$The Pointwise Ergodic Theorem with Weight $\lambda(k)^{2}]$ 
For any ergodic dynamical system $(X, \mathcal{A}, \nu, \tau)$ and any $f \in L^p(\nu)$ where $p \in (1, \infty),$ the following limit converges $\nu$-almost everywhere:
$$\lim_{n\rightarrow \infty} \frac{1}{\sum_{k=1}^{n} \lambda(k)^{2}} \sum_{k=1}^{n} \lambda(k)^{2}f\left(\tau^{k}(x)\right).$$
\end{Th}
\begin{Rem}\label{lastremark}
The sum of squares of $\lambda(k)$ satisfies that 
$$\sum_{k=1}^{n} \lambda(k)^{2} =C_{\Phi}n + O(n^{3/5-1/560+o(1)})$$ for some constant $C_{\Phi}$
(for the details, see \cite{huang2021rankinselberg}). Furthermore, 
the sum of fourth powers of $\lambda(k)$ satisfies $$\sum_{k=1}^{n} \lambda(k)^{4} = C_{\Phi,1}n\log n + C_{\Phi,2}n + O(n^{9/10+o(1)})$$ for some constants $C_{\Phi,1}, C_{\Phi,2}.$
(see \cite{FM}).
Therefore, we cannot apply the automatic dominated ergodic theorem \cite[Lemma 2.1, Lemma 2.2, Lemma 2.4]{CunyWeber} to prove Theorem \ref{hecketheorem}.
\end{Rem}
\noindent 
\subsection{The Piltz divisor function}
To partially answer the problem posed in \cite[Problem 9.4]{CunyWeber}, we use the Piltz divisor function $d_{v}(n),$ defined as
$$d_{v}(n):= \sum_{\substack{m_{1}m_{2}...m_{v}=n \\ m_{i} \in \mathbb{N}}} 1.$$
To obtain our results, we utilize the estimates presented in \cite{MRT1}.
\begin{Th}\label{mainthm}$[$The Pointwise Ergodic Theorem with Weight $d_{v}(k)]$
Let $v \in \mathbb{N}.$  For any ergodic dynamical system $(X, \mathcal{A}, \nu, \tau)$ and any  $f \in L^p(\nu)$ where $p \in (3/2, \infty),$  the following limit converges $\nu$-almost everywhere:
$$\lim_{n\rightarrow \infty} \frac{1}{\sum_{k=1}^{n} d_{v}(k)} \sum_{k=1}^{n} d_{v}(k)f\left(\tau^{k}(x)\right).$$ 
\end{Th}
\subsection{Definitions and notations} Let us define some terms and symbols. The notations $f \ll g$ and $f=O(g)$ indicate that there exist a positive real number $C$ and a real number $x_{0}$ such that for all $x \geq x_0$, 
\begin{equation}
|f(x)| \leq C g(x).
\nonumber \end{equation}
If the constant $C$ depends on a parameter $r$, we denote this dependence by  $f \ll_{r} g $ or $f=O_{r}(g)$.
As $g(x)$ is nonzero, the notation $f(x)=o(g(x))$ means that  $\lim _{x \rightarrow \infty} f(x)/g(x) = 0.$
we use $\hat{f}$ to denote the Fourier transform of $f$, which is defined as $$\hat{f}(y)= \int_{0}^{1} f(x) e^{-2\pi i x y} dx .$$
We use the function $\delta_{k}(x)$, which is defined as follows:
$$\delta_{k}(x)= \begin{cases}1 & \text { if } x = k  \\ 0 & \text{otherwise}. \end{cases}$$
We denote the Fourier transform of the $\lambda(k)^2$-weighted kernel $$\sum_{1 \leq k \leq n} \lambda(k)^{2} \delta_{k}$$ by 
$$\Lambda_{n}(x):= \sum_{1 \leq k \leq n} \lambda(k)^{2}e(kx),$$
and, as indicated in Remark \ref{lastremark}, we have  $$\Lambda_{n}:= \Lambda_{n}(0)=C_{\Phi}n+O_{\Phi}(n^{3/5-1/560+o(1)}).$$
In order to apply the circle method, we use the function
$$\psi_{n,q}(x):= \frac{D_{q}}{\Lambda_{n}} \sum_{m=1}^{n} e(mx)$$
to approximate $\Lambda_{n}(x)$ (normalized) when $x$ is sufficiently close to a reduced rational $a/q.$ $D_{q}$ is a coefficient that is bounded by $1/q^{1 - o(1)}$ (see Lemma \ref{1lemma} for details).  The Fourier inversion of $\psi_{n,q}(x)$ is denoted by
$$\omega_{n,q}(x) := \frac{D_{q}}{\Lambda_{n}} \sum_{m = 1}^{n} \delta_{m}.$$
For the $d_v(k)$-weighted kernel
$$\sum_{1 \leq k \leq n} d_{v}(k) \delta_{k},$$
we denote its Fourier transform by   
$$\mathcal{D}_{n}(x):= \sum_{1 \leq k \leq n} d_{v}(k)e(kx).$$
Using the hyperbola method, we have 
$$\mathcal{D}_{n}:= \mathcal{D}_{n}(0)=\sum_{k=1}^{n} d_{v}(k)= n P_{v}(\log n) + O(n^{1-\frac{1}{v}})$$
where $P_{v}$ is a polynomial of degree $v-1$ (see \cite[Section 1, exercise 2]{IK1}).
To apply the circle method, we use the function
$$\psi_{n,q}^{\ast}(x):= \frac{E_{q,n}}{\mathcal{D}_{n}} \sum_{m=1}^{n} e(mx)$$
for approximating $\mathcal{D}_{n}(x)$ (normalized) when $x$ is sufficiently close to a reduced rational $a/q.$ The coefficient $\displaystyle E_{q,n}$ is bounded by $(\log n)^{v-1} / q^{1-o(1)} $ (see Lemma \ref{lemma41}). We denote the Fourier inversion of $\psi_{n,q}^{\ast}$ by $\displaystyle \omega^{\ast}_{n,q}(x) := \frac{E_{q,n}}{\mathcal{D}_{n}} \sum_{m=1}^{n} \delta_{m}$. 
Later, we will use the smooth function 
$$\eta(x)= \begin{cases}1 & \text { if } x \in\left[-\frac{1}{4}, \frac{1}{4}\right], \\ 0 & \text { if } x \in \mathbb{R} \backslash[-1 / 2,1 / 2] \\ \in C^{\infty} & \text { if } x \in [-1 / 2,1 / 2] \backslash\left[-\frac{1}{4}, \frac{1}{4}\right]\end{cases},$$ and define 
$$\eta_{s}(x):= \eta(4^{s}Mx)$$ for a constant $M.$
Finally, we denote the Fourier inversion of the normalized $\lambda(k)^2$-weighted kernel by $T_{n}(x) := \Lambda_{n}(x)/\Lambda_{n}$.
We approximate $T_{n}(x)$ as 
$$\varphi_{n}(x):= \psi_{n,0}(x)\eta_{0}(x) + \sum_{s=1}^{\infty} \sum_{2^{s-1} \leq q <2^{s}} \sum_{1 \leq a \leq q, (a,q)=1} \psi_{n,q}(x-a/q)\eta_{s}(x-a/q).$$
The Fourier inversion of the normalized $d_{v}(k)$-weighted kernel is denoted by
$$T^{\ast}_{n}(x):= \mathcal{D}_{n}(x)/\mathcal{D}_{n}.$$
We approximate $T^{\ast}_{n}(x)$ as 
$$\varphi^{\ast}_{n}(x):= \psi^{\ast}_{n,0}(x)\eta_{0}(x) + \sum_{s=1}^{\infty} \sum_{2^{s-1} \leq q <2^{s}} \sum_{1 \leq a \leq q, (a,q)=1} \psi^{\ast}_{n,q}(x-a/q)\eta_{s}(x-a/q).$$
Let $\epsilon>0$ be an arbitrarily small positive quantity. We define the major arcs as 
$$M_{n,\mathcal{S}}:= \{x \in [0,1] : |x-a/q| \leq 1/Q_{n,\mathcal{S}} \;\;\textrm{for some } \;\; 0 \leq a \leq q \leq P_{n,\mathcal{S}}\},$$ and the minor arcs as
$$m_{n,\mathcal{S}}:= [0.1] \cap M_{n,\mathcal{S}}^{c}$$
where $$P_{n,\mathcal{S}}=(\log n)^{3\mathcal{S}(1-\epsilon)},\;\; Q_{n,\mathcal{S}}= n/ (\log n)^{2\mathcal{S}(1-\epsilon)} $$ for some $\mathcal{S}>0.$ We set  $\mathcal{S}$ to be a sufficiently large constant for the squares of Hecke eigenvalues, and we set $\mathcal{S}=1$ for the Piltz divisor functions.
We also use the following notations: $I_{\rho} := \{ \lfloor\rho^{n}\rfloor : n \in \mathbb{N} \}$, $\|x\| := \min_{n\in \mathbb{Z}} |x-n|$, and $\phi(n): = |\{1 \leq a \leq q : (a,q)=1\}|$. We denote the set of all prime numbers by $\mathcal{P}$, and we use the common convention that $\epsilon$ denotes an arbitrarily
small positive quantity that may vary from line to line.
\subsection{Sketch of the proof}
We follow the argument in \cite{CunyWeber}, which is based on Bourgain's method. Let $(z_{n})_{n \in \mathbb{Z}}$ be a sequence in $l^{2}(\nu),$ and let $(\mathbb{Z},S)$ be a shift model such that $S(z_{n})_{n \in \mathbb{Z}}=(z_{n+1})_{n \in \mathbb{Z}}.$ Let us fix a measurable dynamical system $(X,\mathcal{A},\nu,\tau)$ and let $f \in L^{2}(\nu).$ 
For convenience, let $(z_{n}):= (z_{n})_{n \in \mathbb{Z}} .$
Define $\overline{(z_{n})}=(z_{-n}).$ Then the shift $S^{k}(z_{n+j})(0) = S^{k}\overline{(z_{-n-j})}(0),$ so we have
\begin{equation}\begin{split}
\frac{\sum_{1 \leq k \leq n} \lambda(k)^{2} S^{k} (z_{n+j})(0)}{ \sum_{1 \leq k \leq n} \lambda(k)^{2} }
&= \frac{\sum_{1 \leq k \leq n} \lambda(k)^{2} S^{k} \overline{(z_{-n-j})}(0)}{\sum_{1 \leq k \leq n} \lambda(k)^{2} }
\\&= \frac{\sum_{1 \leq k \leq n} \lambda(k)^{2} \bar{z}_{-k-j}}{\sum_{1 \leq k \leq n} \lambda(k)^{2} }
\\&= K_{n} \ast \bar{z} (-j)
\end{split}\end{equation}
where 
$$K_{n}:= \frac{\sum_{k=1}^{n} \lambda(k)^{2} \delta_{k}}{\sum_{1 \leq k \leq n} \lambda(k)^{2} }.$$
 By Calderon's transference principle and Bourgain's variational estimates (see \cite[Section 4]{CunyWeber} or \cite[Theorem 3.4]{Bourgaintool}), Theorems \ref{hecketheorem} and \ref{mainthm} can be derived from the following:
\begin{equation}\label{osc}\begin{split}&\sum_{j=1}^{J} \left\| \sup_{N_{j} \leq N \leq N_{j+1}, N \in I_{\rho}} |(K_{N}-K_{N_{j}})\ast g| \right\|^{2}_{l^{2}} = o(J),
\\&\sum_{j=1}^{J} \left\| \sup_{N_{j} \leq N \leq N_{j+1}, N \in I_{\rho}} |(K^{\ast}_{N}-K^{\ast}_{N_{j}})\ast g| \right\|^{2}_{l^{2}} = o(J)\end{split}\end{equation} for every $\rho>1$ and every sequence $(N_{j})$ such that $N_{j+1} \geq 2N_{j}$ with a sufficiently large $J$ depending on $\rho.$ 

Before addressing \eqref{osc}, let us examine some tools. Using the Fourier transform, we have 
$$K_{n} \ast f (j)= \int_{0}^{1} \hat{K}_{n}(t) \hat{f}(t) e(-jt)dt.$$
Using the Hardy-Littlewood maximal inequality  (see \cite[Section 4]{CunyWeber}), we have 
\begin{equation}\label{maximalshift}
\| \sup_{n \geq 1} |\kappa_{n} \ast g | \|_{l^{p}} \ll \|g\|_{l^{p}} \;\; \textrm{for} \;\; p>1,
\end{equation}
where $$\kappa_{n}:=\frac{\sum_{k=1}^{n} \delta_{k}}{n}.$$
This is a maximal inequality for the Ces\'aro kernel. By applying the above inequality, we have
\begin{equation}\label{maximal}\begin{split}&\|w_{n,q}\ast g|\|_{l^{p}} \ll |D_{q}| \|g\|_{l^{p}} \ll_{\Phi,\epsilon} q^{-1+\epsilon} \|g\|_{l^{p}}, \\&
\|w^{\ast}_{n,q}\ast g|\|_{l^{p}} \ll \frac{|E_{q,n}|}{(\log n)^{v-1}} \|g\|_{l^{p}} \ll_{\epsilon} q^{-1+\epsilon} \|g\|_{l^{p}}.\end{split}\end{equation}
Cuny and Weber used Fourier analytic methods, building upon the work of Wierdl \cite{Wierdl2}, to establish results in \cite[Proposition 7.2, Theorem 7.5]{CunyWeber}. By setting $\tau$ to a value arbitrarily close to 1 and choosing  sufficiently large constants $\gamma$ and $S$, the following proposition is derived from their work \cite[Proposition 7.2, Theorem 7.5]{CunyWeber}. Furthermore, Proposition \ref{C1} implies Theorem \ref{hecketheorem}. The proof of Theorem \ref{mainthm} follows the same steps, except that we replace $\psi_{n,q}, w_{n,q}, T$ with $\psi^{\ast}_{n,q}, w^{\ast}_{n,q}, T^{\ast}$ respectively and set $\mathcal{S}=1$ in Proposition \ref{C1}. Therefore, the remaining task is to prove the assumptions in Proposition \ref{C1}.
\begin{Prop}\label{C1} Assume the following:
 \begin{enumerate}
    
    \item For any $x$ in $M_{n,\mathcal{S}}$ and for sufficiently large values of $\mathcal{S},$
    \begin{equation}\label{Tnmajor} |T_{n}(x)- \psi_{n,q}(x-a/q)| \ll_{\Phi,\epsilon} (\log n)^{-\mathcal{S}+\epsilon} \;\; \end{equation} where $(a,q)$ is either $(0,0)$ or a pair of natural numbers such that
    $1\leq a \leq q \leq P_{n,\mathcal{S}}$ and $|x-a/q| \leq 1/Q_{n,\mathcal{S}}.$
    
    \item For any $x \in m_{n,\mathcal{S}}$ and for sufficiently large values of $\mathcal{S},$ \begin{equation}\label{Tnminor} |T_{n}(x)| \ll_{\Phi,\epsilon} (\log n)^{-\mathcal{S}+\epsilon}.  \;\;  \end{equation}
    \item For any $\rho>1$ and any sequence $\left(N_{j}\right)$ with $N_{j+1}>2N_{j}$,
\begin{equation}\label{pointwise} 
\sum_{j \geq 1} \left\| \sup_{\substack{N_{j} \leq N \leq N_{j+1}\\ N \in I_{\rho}}} |(w_{N,q}-w_{N_{j},q})\ast g \right\|_{l^{2}}^{2} \ll_{\Phi,\epsilon} q^{-1+\epsilon}\|g\|_{l^{2}}^{2}.
\end{equation}
\end{enumerate}
Then, for any $\rho>1$ and any sequence $\left(N_{j}\right)$ with $N_{j+1}>2N_{j}$,
\begin{equation}
\sum_{j=1}^{J} \left\| \sup_{\substack{N_{j} \leq N \leq N_{j+1}\\ N \in I_{\rho}}} |(K_{N}-K_{N_{j}})\ast g| \right\|_{l^{2}}^{2} = o(J),
\end{equation} for large $J$ depending on $\rho.$
\end{Prop}
\begin{Rem}
Proposition \ref{C1}  still holds when you replace     $\psi_{n,q}, w_{n,q}, T_{n}$ with $\psi^{\ast}_{n,q}, w^{\ast}_{n,q}, T_{n}^{\ast},$ respectively and set $\mathcal{S}=1.$
\end{Rem}
To prove the assumptions in Proposition \ref{C1}, we transform the additive twist $e\left(an/q \right)$ into multiplicative twists $\chi(n)$ for some Dirichlet characters $\chi$. We then use certain analytic properties of $L$-functions twisted by Dirichlet characters to obtain the average values of $\lambda \left(q_0 n\right) \chi(n)$. After attaching $e(n \beta)$ to $\lambda (n)^2 e\left(an/q\right),$ through integration by parts, we have
$$
\sum_{1 \leq k \leq  n} \lambda (k)^2 e\left(\frac{a}{q} k\right) e(\beta k) \sim \int_1^{n} \sum_{q=q_0 q_1} \frac{\mu\left(q_1\right)}{\phi\left(q_1\right) q_0} w_{\Phi, q_0, q_1} e(\beta x) d x
$$ where $w_{\Phi, q_0, q_1}$ is a weighted sum over primitive characters. By applying the Euler-Maclaurin summation formula, we obtain an approximate Fourier series for $T_n(x)$.
\begin{Rem} 
For any $\rho>1$ and any sequence $\left(N_{j}\right)$ with $N_{j+1}>2N_{j}$, the oscillation inequality 
\begin{equation}\label{difference}\begin{split}
\sum_{j \geq 1} \left\| \sup_{\substack{N_{j} \leq N \leq N_{j+1}\\ N \in I_{\rho}}} |(\kappa_{N}-\kappa_{N_{j}})\ast g |\right\|_{l^{2}}^{2} \ll \|g\|_{l^{2}}
\end{split}\end{equation}
holds (see \cite{Wierdl}). We will apply the above oscillation inequality to prove \eqref{pointwise}.
\end{Rem}
\noindent In Corollary \ref{31}, we establish the validity of \eqref{Tnmajor}, while in Corollary \ref{32}, we prove \eqref{Tnminor}. Additionally, Proposition \ref{34} provides the proof for \eqref{pointwise}. Furthermore, in Proposition \ref{Prop44}, we establish the $d_{v}(k)$ version of \eqref{Tnmajor} and \eqref{Tnminor}. The corresponding $d_{v}(k)$ version of \eqref{pointwise} is proven in the same proposition. Lastly, in the final section, we present similar results for other arithmetic functions.
\begin{Rem}
By \cite[Proposition 7.2]{CunyWeber}, Theorems \ref{hecketheorem} and \ref{mainthm} hold for any function $f \in L^{p}(\nu)$ with $p \in (1 + 1/2\mathcal{S},\infty).$ For Hecke eigenvalue squares, we can choose any sufficiently large constant $\mathcal{S}$, since the $L$-functions associated with Hecke cusp forms has a simple pole at $s=1$. The $v$th power of the Riemann zeta function is associated with $d_v(n)$, which has a pole of order $v$ at $s=1$. This restricts us to choosing $\mathcal{S}=1$ (see \eqref{difference2}).

\noindent Proving or disproving the case for $p=1$ seems challenging.
Buczolich \cite{Buczolich} demonstrated that in the case of prime divisor weights, Theorems \ref{hecketheorem} and \ref{mainthm} hold for $q=1.$ To do so, he established the ergodic maximal inequality over integers (for further details, see \cite[Chapter 3]{Wierdl2}). Additionally, he used the fact that the average of the $k$th power of prime divisor weights is equal to the $k$th power of the average of prime divisor weights, and these averages are relatively easy to control for any natural number $k$.
\noindent In contrast, the average of the $k$th power of the Piltz divisor function (and Hecke eigenvalue squares) is greater than the $k$th power of the average when $k$ is greater than or equal to $2$. Furthermore, controlling the average of Piltz divisor functions and Hecke eigenvalue squares over short intervals is more challenging compared to the average of prime divisor weights. Consequently, we cannot directly apply the arguments used in the proof for the prime divisor weight case.
\end{Rem}
 By following the proof of \cite[Lemma 7.1]{CunyWeber}, it is easy to see that for any $n \geq 2$,  
\begin{equation}\label{newequation} |T_{n}(x)-\varphi_{n}(x)| \ll_{\Phi,\epsilon} (\log n)^{-\mathcal{S}+\epsilon} \;\; \textrm{for any} \;\; x \in [0,1] \;\;\textrm{and for sufficiently large values of} \;\;  \mathcal{S}.\end{equation} The inequality stated in \cite[Lemma 7.1]{CunyWeber} was utilized to establish (\cite[Proposition 7.2]{CunyWeber}), which corresponds to Proposition \ref{C1}. Furthermore, the inequality was also applied to demonstrate that $d_{2}(k)$ and $\theta(n)$ are good weighting functions for the dominated ergodic theorem, where $\theta(n)$ is the multiplicative function counting the number of squarefree divisors of $n$ (see \cite{CunyWeber}). In Proposition \ref{33}, for the sake of completeness, we prove \eqref{newequation} to demonstrate that $\lambda(k)^2$ serves as a suitable weighting function for the dominated ergodic theorem.  

\section{Lemmas}
In this section, we introduce some useful lemmas.
\begin{Lemma}\label{TEMF}$[$The Euler-Maclaurin Formula$]$
Let $f:[a,b] \rightarrow \mathbb{C}$ be a continuously differentiable function. Then, we have
$$ \sum_{a<n\leq b} f(n) = \int_{a}^{b} f(x)dx + O\left(\int_{a}^{b}
\left|f'(x)\right|dx + \left|f(b)-f(a)\right|\right).$$\end{Lemma}
\begin{proof} See \cite[Lemma 4.1]{IK1}.
\end{proof}
\begin{Lemma}\label{logsum} For any $v \geq 2$ and any integer $n>2,$
\begin{equation}\begin{split}\sum_{k=1}^{n} (\log k)^{v-1}= n(\log n)^{v-1} + O\left((v-1)n (\log n)^{v-2}\right). \end{split}\end{equation} 
\end{Lemma}
\begin{proof}
Using Lemma \ref{TEMF}, we see that 
$$\sum_{k=2}^{n} (\log k)^{v-1}-\int_{1}^{n} (\log x)^{v-1} dx = O\left( (v-1)\int_{1}^{n} \left(\log x
\right)^{v-2} dx + (\log n)^{v-1}\right).$$
\end{proof}
\section{Proof of Theorem \ref{hecketheorem}}
In this section , we assume that $\mathcal{S}$ is a sufficiently large constant.
\noindent First, let us consider $\Lambda_{n}(a/q).$ This will be useful in obtaining $\Lambda_{n}(a/q+\beta).$  
\begin{Lemma}\label{1lemma}
For any $1 \leq a < q$ such that $(a,q)=1$ or $a=0,q=1$,
\begin{equation}
    \Lambda_{n}(a/q)= D_{q}n +O_{\Phi} \left(q^{5}n^{3/4+\epsilon}\right)
\end{equation}
where
$$D_{q}:=\sum_{q_{0}q_{1}=q} \frac{\mu(q_{1})}{\phi(q_{1})q_{0}}w_{\Phi,q_{0},q_{1}},$$ and 
$$ w_{\Phi, q_0, q_1}=C_{\Phi} \prod_{p \mid q}\left(\frac{p-1}{p+1}\right)\left(1-\frac{\lambda(p)^2-2}{p}+\frac{1}{p^2}\right) \prod_{p^l \| q_0}\left(\lambda\left(p^l\right)^2+\lambda\left(p^{l+1}\right)^2 p^{-1}+\ldots\right)$$
where $p^l \| q_0$ means that $p^{l} |q_{0},$ but $p^{l+1} \nmid q_{0}.$
\end{Lemma}
\begin{proof}
See \cite[Remark 3.5, Remark 3.9]{Kim2022s}. 
\end{proof}
\noindent By applying Deligne's bound $\lambda(p) \leq 2$ (see \cite{Del1} or \cite[Section 14.9]{IK1}), it is easy to show that  
$$
\left|D_q\right| \ll_\Phi  \frac{d_2(q)(\log q)^7}{q}
$$
(see \cite[(3.9)]{Kim2022s}). 
We are ready to get the approximate Fourier series $\psi_{n,q}(x)$ of $T_{n}(x).$ 
\begin{Lemma}\label{2lemma}
Assume the conditions in Lemma \ref{1lemma}. Suppose that $|x-\frac{a}{q}|\leq 1/Q_{n,\mathcal{S}}.$ Then

\begin{equation}
\Lambda_{n}(x)= D_{q} \sum_{m=1}^{n} e\left((x-a/q) m\right) +O_{\Phi,\epsilon}\left(n^{3/4+\epsilon}q^{5+\epsilon}\left(1+nQ_{n,\mathcal{S}}^{-1}\right)\right) 
\nonumber\end{equation}
\end{Lemma}
\begin{proof}
Let $\beta= x-a/q.$
By \cite[Lemma 3.8]{Kim2022s},
$$
\sum_{1 \leq m \leq n} \lambda(m)^2 e\left(\frac{a}{q} m\right) e(\beta m) 
=\int_0^{n} D_{q} e(\beta x) d x+O_\epsilon\left((|\beta|+1/n) n^{7/4+\epsilon/8} q^{5+\epsilon/2}\right).
$$
Using Lemma \ref{TEMF}, we see that  
\begin{equation}\begin{split}\int_0^{n} e(\beta x) d x- \sum_{m=1}^{n} e(\beta m)  &=O\left(1+\int_{0}^{n} |\beta| dx \right)\\&= O\left(1+|\beta| n  \right).\end{split}\end{equation}
Therefore, we have 
\begin{equation}
\Lambda_{n}(x)= D_{q} \sum_{m=1}^{n} e\left((x-a/q) m\right) +O_{\Phi,\epsilon}\left(n^{3/4+\epsilon}q^{5+\epsilon}\left(1+nQ_{n,\mathcal{S}}^{-1}\right)\right). 
\nonumber\end{equation}
\end{proof}
\noindent Recall that  $$M_{n,\mathcal{S}}:= \{x \in [0,1] : |x-a/q|\leq 1/Q_{n,\mathcal{S}} \;\;\textrm{for some } \;\; 0 \leq a \leq q \leq P_{n,\mathcal{S}}\},$$
$$m_{n,\mathcal{S}}:= [0.1] \cap M_{n,\mathcal{S}}^{c}.$$

\begin{Cor}\label{31}
Assume that $x \in M_{n,\mathcal{S}},$ $|x-a/q| \leq \frac{1}{Q_{n,\mathcal{S}}}$ where $0\leq a \leq q \leq P_{n,\mathcal{S}}.$ Then
$$|T_{n}(x)-\psi_{n,q}(x-a/q)| \ll_{\Phi,\epsilon} P_{n,\mathcal{S}}^{5+\epsilon}n^{-1/4+\epsilon}\left(1+nQ_{n,\mathcal{S}}^{-1}\right).$$
\end{Cor}
\begin{proof} The proof directly follows from Lemma \ref{2lemma}.
\end{proof}

\noindent For the minor arc estimates, we apply Lemma \ref{1lemma}.
\begin{Lemma}\label{3lemma}
Assume that for all $1 \leq a \leq q \leq P_{n,\mathcal{S}},$ we have $|x-a/q|>\frac{1}{Q_{n,\mathcal{S}}}.$ 
Then
$$\Lambda_{n}(x)=O_{\Phi,\epsilon}\left(P_{n,\mathcal{S}}^{\epsilon-1} n + n^{2}/P_{n,\mathcal{S}}Q_{n,\mathcal{S}} + P_{n,\mathcal{S}}^{5+\epsilon} n^{3/4+\epsilon}\right).$$
\end{Lemma}
\begin{proof}
The proof follows from the proof of \cite[Lemma 5.4]{CunyWeber}. By the Dirichlet principle, $|x-a/q|\leq 1/qQ_{n,\mathcal{S}}$ for some $1 \leq a \leq q \leq Q_{n,\mathcal{S}}.$ Because of our assumption, $q>P_{n,\mathcal{S}}.$ 
Since $|e(kx)- e(ka/q)| \leq 2\pi k|x-a/q| \leq 2\pi k/P_{n,\mathcal{S}}Q_{n,\mathcal{S}},$ we have 
\begin{equation}\begin{split}|\Lambda_{n}(x)| &\leq |\Lambda_{n}(a/q)|+ \frac{2\pi}{P_{n,\mathcal{S}}Q_{n,\mathcal{S}}} \sum_{1 \leq k \leq n} k \lambda(k)^{2} 
\\&\ll_{\Phi,\epsilon} |D_{q}|n +P_{n,\mathcal{S}}^{5+\epsilon} n^{3/4+\epsilon} +n^{2}/P_{n,\mathcal{S}}Q_{n,\mathcal{S}}.   \end{split}\end{equation}
Since $q>P_{n,\mathcal{S}},$ we have $D_{q} \ll_{\Phi,\epsilon} P_{n,\mathcal{S}}^{\epsilon-1}.$
The proof is completed. \end{proof}
\begin{Cor}\label{32}
For any $x\in m_{n,\mathcal{S}},$ 
\begin{equation}\label{minor} |T_{n}(x)|=O_{\Phi,\epsilon}\left(P_{n,\mathcal{S}}^{\epsilon-1} + n/P_{n,\mathcal{S}}Q_{n,\mathcal{S}} + P_{n,\mathcal{S}}^{5+\epsilon} n^{-1/4+\epsilon}\right).\end{equation}
\end{Cor}
\begin{proof}
The proof directly follows from Lemma \ref{3lemma}.
\end{proof}
\begin{Lemma}\label{4lemma} For any $q \in \mathbb{N},$
$$|\psi_{n,q}(x)| \ll_{\Phi,\epsilon} (q+1)^{-1+\epsilon} \min\left(1, 1/n\|x\|\right).$$ 
\end{Lemma}
\begin{proof}
Note that $|D_{q}|\ll q^{-1+\epsilon}.$ By using the elementary bound
$$\sum_{m=1}^{n} e (\beta n) \leq 2 \min (n, 1/\|\beta\|), $$
the proof is completed.
\end{proof}
\begin{Prop}\label{34} Assume  $n \geq 1, q \geq 0.$ Then 
$$\sum_{j \geq 1} \left\| \sup_{\substack{N_{j} \leq N \leq N_{j+1}\\ N \in I_{\rho}}} |(w_{N,q}-w_{N_{j},q})\ast g \right\|_{l^{2}}^{2} \ll_{\Phi,\epsilon} (q+1)^{-1+\epsilon}\|g\|_{l^{2}}^{2}.$$
\end{Prop}
\begin{proof}
Since $$D_{q} \ll_{\Phi,\epsilon} (q+1)^{-1+\epsilon},$$ 
the proof directly follows from \eqref{difference}.
\end{proof}
\begin{Prop}\label{33}
Let $x\in [0,1].$ 
Then 
\begin{equation}
|T_{n}(x)-\varphi_{n}(x)| \ll_{\Phi,\epsilon} (\log n)^{-\mathcal{S}+\epsilon}.
\end{equation} 
\begin{proof} Recall that
$$\varphi_{n}(x):= \psi_{n,0}(x)\eta_{0}(x) + \sum_{s=1}^{\infty} \sum_{2^{s-1} \leq q <2^{s}} \sum_{1 \leq a \leq q, (a,q)=1} \psi_{n,q}(x-a/q)\eta_{s}(x-a/q).$$
The proof follows from the proof of \cite[Lemma 7.1]{CunyWeber}. 
Let us consider the major arcs $M_{n,\mathcal{S}}$ first. Let $x \in M_{n,\mathcal{S}}.$ If $x=0,$ then $|x-a_{s}/q_{s}|=a_{s}/q_{s}$ where $2^{s-1} \leq q_{s} \leq 2^{s}, 1 \leq a_{s} \leq q_{s},$ and $(a_{s},q_{s})=1.$  
Since $a_{s}/q_{s} \geq 1/2^{s},$
$$\eta_{s}(-a/q)=0.$$ Therefore, 
$$\varphi_{n}(0)=\psi_{n,0}.$$ 
If $x \in M_{n,\mathcal{S}}$ and $x \neq 0,$ then there's a pair $(a,q)$ such that $1 \leq a \leq q \leq P_{n,\mathcal{S}},$ $|x-\frac{a}{q}|\leq \frac{1}{Q_{n,\mathcal{S}}}.$ 
Now let's consider the summation over $s.$ 
Assume that $s' \geq 1$ satisfies $2^{s'-1} \leq q < 2^{s'}.$ 
Then 
$$|x-a/q| \leq 1/Q_{n,\mathcal{S}} \leq 1/8MP_{n,\mathcal{S}}^{2} \leq 1/2M4^{s'} \Rightarrow \eta_{s'}(x-a/q ) =1.$$
Therefore, we have 
$$ \psi_{n,q}(x-a/q)\eta_{s'}(x-a/q)=  \psi_{n,q}(x-a/q).$$
For $(a_{s'},q_{s'}) \neq (a,q)$ where $2^{s'-1} \leq q_{s'} <2^{s'}, 1 \leq a_{s'} \leq q_{s'},$ and $ (a_{s'},q_{s'})=1,$ we have  
\begin{equation}\label{work}|x-a_{s'}/q_{s'}| \geq |a/q-a_{s'}/{q_{s'}}|-|x-a/q| \geq 1/P_{n,\mathcal{S}}^{2}-1/Q_{n,\mathcal{S}} \geq (8M-1)/Q_{n,\mathcal{S}}.\end{equation}
Using Lemma \ref{4lemma}, we see that  
$$\psi_{n,q_{s'}}(x-a_{s'}/q_{s'}) \ll_{\Phi,\epsilon} q_{s'}^{-1+\epsilon} (\log n)^{-\mathcal{S}+\epsilon}.$$
If $s \neq s',$ then for any pair $(a_{s},q_{s})$ such that 
$1 \leq a_{s} \leq q_{s} \leq P_{n,\mathcal{S}}, (a_{s},q_{s})=1,$ we have
$$|x-a_{s}/q_{s}| \geq |a/q-a_{s}/{q_{s}}|-|x-a/q| \geq  (8M-1)/Q_{n,\mathcal{S}}.$$
Therefore,
\begin{equation}\label{equa1}\begin{split}|\varphi_{n}(x)-T_{n}(x)| &\ll_{\Phi,\epsilon} (\log n)^{-\mathcal{S}+\epsilon} +  (\log n)^{-\mathcal{S}+\epsilon} \sum_{s=1}^{\log_{2} P_{n,\mathcal{S}}} 2^{-(1-\epsilon) s} +  \sum_{s= \lfloor \log_{2} P_{n,\mathcal{S}}\rfloor }^{\infty} 2^{-(1-\epsilon) s}
\\&\ll_{\Phi,\epsilon} (\log n)^{-\mathcal{S}+\epsilon}. \end{split}\end{equation}
Now let us consider $x \in m_{n,\mathcal{S}}.$ By the assumption, we see that 
$$|\psi_{n,0}(x)| \ll_{\Phi} Q_{n,\mathcal{S}}/n \ll_{\Phi,\epsilon} (\log n)^{-\mathcal{S}+\epsilon}.$$
Using \eqref{minor}, we have 
$$|T_{n}(x)- \varphi_{n}(x)| \leq |T_{n}(x)|+|\varphi_{n}(x)| \ll_{\Phi} (\log n)^{-\mathcal{S}+\epsilon} +   |\varphi_{n}(x)|.$$
By the Dirichlet principle, $|x-a/q|\leq 1/qQ_{n,\mathcal{S}}$ for some $0 \leq a \leq q \leq Q_{n,\mathcal{S}}.$ If $(a',q')=1,$ $a/q \neq a'/q',$ then by using a similar argument as in \eqref{work},
$$\psi_{n,q'}(x-a'/q') \ll_{\Phi,\epsilon} (q'+1)^{-(1+\epsilon)}(\log n)^{-\mathcal{S}+\epsilon}$$ when 
$q' \leq (\log n)/2.$ When $a'=a$ and $q'=q,$
$$ \psi_{n,q'}(x-a'/q') \ll_{\Phi,\epsilon} (q'+1)^{-1+\epsilon}\ll_{\Phi,\epsilon} (\log n)^{-\mathcal{S}+\epsilon}$$ since $q'=q>P_{n}.$
By using a similar argument as in \eqref{equa1}, we have 
$$ \varphi_{n}(x)\ll_{\Phi,\epsilon} (\log n)^{-\mathcal{S}+\epsilon}.$$ 
 \end{proof}
\end{Prop}

\section{Propositions for $d_{v}(k)$} 
In Sections 4 and 5, we will use the simplified notations $P_n$, $Q_n$, $M_n$, and $m_n$ instead of $P_{n,1}$, $Q_{n,1}$, $M_{n,1}$, and $m_{n,1}$, respectivelely.
\begin{Lemma}\label{lemma41}\cite[Proposition 4.2]{MRT1}
Let $A,B,B'>0, v \geq 2,$ and $ n \geq 2.$ Let $x=a/q+\beta$ for some $1 \leq q \leq (\log n)^{B}, (a,q)=1,$ and $|\beta| \leq \frac{(\log n)^{B'}}{n}.$ 
Then
$$\sum_{k \leq n} d_{v}(k)e(kx) = \int_{0}^{n} P_{v,q}(x) e(\beta x) dx + O_{v,A,B} \left(n(\log n)^{-A}\right)$$
where $$P_{v,q}(x)= \sum_{q=q_{0}q_{1}}\frac{\mu(q_{1})}{\phi(q_{1})q_{0}} \frac{d}{dx} \left(\text{Res}_{s=1} \left( \frac{(x/q_{0})^{s}}{s}\sum_{\substack{n\geq 1\\ (n,q_{1})=1}} \frac{d_{v}(q_{0}n)}{n^{s}}\right)\right).$$
Note that $P_{v,q}(x)$ is a polynomial of degree $v-1$ in $\log x.$
\end{Lemma}
\begin{proof} See \cite[Proposition 4.2]{MRT1}, where the proof is done through partial integration. The only modification required for this is adjusting the range of the summation, which can be easily accomplished.

\end{proof} 
\noindent We set $B=1-\epsilon, B'=2(1-\epsilon)$ and $A=2022.$ 
Since 
$$\sum_{\substack{n\geq 1\\ (n,q_{1})=1}} \frac{d_{v}(q_{0}n)}{n^{s}} \leq q_{0}^{s}\sum_{n\geq 1} \frac{d_{v}(n)}{n^{s}} \;\; \textrm{for} \;\; s>1,$$
we have
$$P_{v,q}(x) = \sum_{i=1}^{v} C_{q,i} (\log x)^{v-i}$$ 
for some constants $C_{q,i}$ and satisfy $|C_{q,i}|\ll_{\epsilon} q^{-1+\epsilon}$.
\begin{Prop}\label{Prop42}
Let $x \in M_{n},$ $|x-a/q| \leq \frac{1}{Q_{n}}$ where $0\leq a \leq q \leq P_{n}.$ Then, we have
$$|T^{\ast}_{n}(x)-\psi^{\ast}_{n,q}(x-a/q) |\ll_{v,\epsilon} (\log n)^{-1+\epsilon}.$$ 
For $x \in m_{n},$  we have 
 \begin{equation} |T^{\ast}_{n}(x)| \ll_{\epsilon} (\log n)^{-1+\epsilon}. \nonumber\end{equation}

\end{Prop}
\begin{proof} Let $\beta=x-a/q.$ Let us prove the second inequality first. 
Using a similar method as in Lemma \ref{3lemma},
\begin{equation}\begin{split}|\mathcal{D}_{n}(x)| &\leq |\mathcal{D}_{n}(a/q)|+ \frac{2\pi}{P_{n}Q_{n}} \sum_{1 \leq k \leq n} k d_{v}(k) 
\\&\ll |E_{q,n}|n + n/ (\log n)^{-2022} +n^{2}(\log n)^{v-1}/P_{n}Q_{n}  
\\&\ll_{\epsilon}  q^{-1+\epsilon} (\log n)^{v-1} n + n/ (\log n)^{-2022} +n^{2}(\log n)^{v-1}/P_{n}Q_{n}   
\end{split}\end{equation} 
where 
$E_{q,n}:= \sum_{i=1}^{v} C_{q,i} (\log n)^{v-i}.$ 
The proof of the second inequality is completed since we have $q > P_{n} = (\log n)^{3(1-\epsilon)}$ and $Q_{n} = n/(\log n)^{2(1-\epsilon)}$. Now let's consider the first inequality. 
By Lemma \ref{lemma41}, we are left to prove that 
$$\int_{0}^{n} P_{v,q}(x) e(\beta x)dx = E_{q,n} \sum_{k=1}^{n} e(k\beta) +O\left(n(\log n)^{v-2+\epsilon})\right). $$
Using Lemma \ref{TEMF}, we see that 
\begin{equation}\begin{split}\int_{0}^{n} P_{v,q}(x) e(\beta x)dx&= \sum_{k=1}^{n} \sum_{ i=1}^{v} C_{q,i} (\log k)^{v-i}e(k \beta) \\&\;\;\;+ O_{\epsilon,v}\left(q^{-1+\epsilon}(\log n)^{v-1}/2+ q^{-1+\epsilon}n(\log n)^{v-1}\left(\frac{1}{\log n}+ 1/Q_{n}\right)\right).\end{split}\nonumber\end{equation}
Since $C_{q,i} \ll_{\epsilon} q^{-1+\epsilon},$
\begin{equation}\label{difference2}\begin{split}&\sum_{k=1}^{n} \sum_{i=1}^{v} C_{q,i} \left((\log n)^{v-i} - (\log k)^{v-i}\right)e( k\beta) 
\\&\ll_{\epsilon}  \sum_{i=1}^{v} \sum_{k=1}^{n}  q^{-1+\epsilon} \left((\log n)^{v-i} - (\log k)^{v-i}\right)
\\&\ll_{\epsilon}  q^{-1+\epsilon}\left( n\sum_{i=1}^{v}(\log n)^{v-i}- \sum_{i=1}^{v} \sum_{k=1}^{n} (\log k)^{v-i}\right).
\end{split}\end{equation}
Using Lemma \ref{logsum}, we have 
\begin{equation}\begin{split} \sum_{i=1}^{v} \sum_{k=1}^{n} (\log k)^{v-i} &= n\sum_{i=1}^{v}(\log n)^{v-i} + O\left(n\sum_{i=1}^{v} (v-i) (\log n)^{v-i-1}\right) \\&= 
n\sum_{i=1}^{v} (\log n)^{v-i} + O\left(v^{2} n (\log n)^{v-2}\right).\end{split}\nonumber\end{equation}
Therefore, we see that 
$$ q^{-1+\epsilon}\left( n\sum_{i=1}^{v}(\log n)^{v-i}- \sum_{i=1}^{v} \sum_{k=1}^{n} (\log k)^{v-i}\right)\ll_{\epsilon,v} nq^{-1+\epsilon} (\log n)^{v-2}.$$
\end{proof}
\begin{Prop}\label{Prop44} Let $n \geq 1, q \geq 0.$ Then 
$$\sum_{j \geq 1} \left\| \sup_{\substack{N_{j} \leq N \leq N_{j+1}\\ N \in I_{\rho}}} |(w^{\ast}_{N,q}-w^{\ast}_{N_{j},q})\ast g \right\|_{l^{2}}^{2} \ll_{\epsilon} (q+1)^{-1+\epsilon}\|g\|_{l^{2}}^{2}.$$
\end{Prop}
\begin{proof}
Since $$E_{q,n}/\mathcal{D}_{n} \ll_{\epsilon} n^{-1}(q+1)^{-1+\epsilon},$$  
the proof directly follows from \eqref{difference}.
\end{proof}
\begin{Rem}
Using a similar argument as in the proof of Proposition \ref{33}, it is easy to prove that 
\begin{equation}\label{wholebound}
|T^{\ast}_{n}(x)-\varphi^{\ast}_{n}(x)| \ll_{\epsilon} (\log n)^{-1+\epsilon} \;\; \textrm{for any} \;\; x \in [0,1].
\end{equation} 
\end{Rem}

\section{General cases}
\noindent In this section, we will demonstrate that the methods used in Section 3 can be applied to various arithmetic functions associated with classical $L$-functions. One may also notice that the arguments presented in Proposition \ref{C1} for the proof of Theorem \ref{hecketheorem} can be extended to other arithmetic functions.
\begin{Prop}\label{generalprop}
Let $$\mathcal{A}(s):=\sum_{m=1}^{\infty} \alpha(m)m^{-s}=\prod_{p \in \mathcal{P}}\left(1+\frac{\alpha(p)}{p^{s}}+\frac{\alpha(p^2)}{p^{2s}}+...\right)$$ be a function which satisfies the following properties:
\begin{enumerate} 
    \item There exists a natural number $k \in \mathbb{N}$ such that $$0 \leq \alpha(m) \leq d_{2}(m)^{k} \;\;\textrm{for all}\;\; m \in \mathbb{N}.$$
     \item $\mathcal{A}(s)$ has analytic continuation to the whole complex plane, where it is holomorphic except for a pole $s=1$ of order $\varkappa \geq 1.$  
    \item  For any non-principal Dirichlet character  $\chi$, $\mathcal{A}(\chi,s):=\sum_{m=1}^{\infty} \alpha(m)\chi(m)m^{-s}$ has analytic continuation to the whole complex plane where it is holomorphic.
    \item There exist natural numbers $b_{1},b_{2}>3$ such that for any $\sigma \in [\frac{1}{2},1],$ 
    \begin{equation}\begin{split}
    &\mathcal{A}(\sigma+it) \ll (1+|t|)^{b_{1}}, 
    \\&\mathcal{A}(\chi,\sigma+it):= \sum_{m=1}^{\infty} \alpha(m)\chi(m)m^{-s} \ll q^{b_{2}}(1+|t|)^{b_{1}}
    \end{split}\nonumber\end{equation}
    where $\chi$ is a non-principal Dirichlet character $(\textrm{mod} \;\;q)$.
\end{enumerate}
Let $n$ be sufficiently large. Then for any $1 \leq a < q$ such that $(a,q)=1$ or $a=0,q=1$,
\begin{enumerate}[label=(\alph*)]
    \item  $$ \sum_{m=1}^{n} \alpha(m) e(am/q) = nF_{q}(n) + O(q^{b_{2}+1}n^{4/5})$$ 
where $F_{q}(n):= \sum_{i=1}^{\varkappa} \mathcal{C}_{q,i} (\log n)^{\varkappa-i}$ for some constants $\mathcal{C}_{q,i} \ll_{\epsilon} q^{-1+\epsilon}.$
    \item For $|x-a/q|\leq 1/Q_{n}$ where $q \leq P_{n},$
     $$ \sum_{m=1}^{n} \alpha(m) e((x-a/q)m) = G_{q}(n)\sum_{m=1}^{n} e((x-a/q) m) + O_{\epsilon}\left(q^{-1+\epsilon}n(\log n)^{\varkappa-2}\right)$$
     where  $G_{q}(n):= \sum_{i=1}^{\varkappa} \mathcal{C}'_{q,i} (\log n)^{\varkappa-i}$ for some constants $\mathcal{C}'_{q,i} \ll_{\epsilon} q^{-1+\epsilon}.$
     \item  If, for every $1 \leq a \leq q \leq P_{n}$, $|x-a/q| > 1/Q_{n}$, then
$$\sum_{m=1}^{n} \alpha(m) e(mx)=O_{\epsilon}\left(P_{n}^{\epsilon-1} n (\log n)^{\varkappa-1} + n^{2}(\log n)^{\varkappa-1}/P_{n}Q_{n} + P_{n}^{b_{2}+1+\epsilon} n^{4/5}\right).$$
\end{enumerate}

\end{Prop}
\begin{Rem} The Euler product of $\mathcal{A}(s)$ ensures that $\alpha(m)$ is multiplicative. In general,  classical $L$-functions satisfy conditions $\textit{(2),(3)},$ and $\textit{(4)}$ (see \cite[Chapter 5]{IK1}). For automorphic $L$-functions, condition $\textit{(1)}$ is related to the Ramanujan conjecture (see \cite[page 99]{IK1}). Note that the $L$-functions for $\lambda(m)^{2}$, and $d_{v}(m)$ satisfy the conditions in Proposition \ref{generalprop}. 
\end{Rem}
\noindent 
Using similar arguments as in Section 3, we prove the above proposition and the following lemmas. 
\begin{Lemma}\label{Lemma61} Assume the conditions in Proposition \ref{generalprop}.
Then 
\begin{equation}\begin{split}\sum_{m=1}^{n} \alpha(m) e(am/q)&=\sum_{q=q_{0}q_{1}} \frac{\mu(q_{1})}{\phi(q_{1})} \sum_{m_{1} \leq n/q_{0} \atop (m_{1},q_{1})=1} \alpha(q_{0}m_{1})
\\& \;\; + O\left(\sum_{q=q_{0}q_{1}}q_{1}^{\frac{1}{2}+\epsilon} \left| \sum_{m_{1} \leq n/q_{0}} \alpha(q_{0}m_{1})\chi(m_{1})\right|\right). 
\end{split}\nonumber\end{equation}
\end{Lemma}
\begin{proof} The proof follows from the proof of \cite[Proposition 4.2]{MRT1}.
By considering the common divisor $q_{0}$ of $q$ and $m$, we have
$$\sum_{m=1}^{n} \alpha(m) e(am/q)= \sum_{q=q_{0}q_{1}} \sum_{m_{1} \leq n/q_{0}, \atop (m_{1},q_{1})=1} \alpha(q_{0}m_{1}) e(am_{1}/q_{1}).$$
By the orthogonality of $\chi,$ it is easy to see that, for any $m_{1}, q_{1} \in \mathbb{N},$
\begin{equation}\label{rationalidentity}
e(am_{1}/q_{1}) 1_{(m_{1},q_{1})=1} = \frac{1}{\phi(q_{1})} \sum_{\chi \;\textrm{mod} \; q_{1}} \chi(a)\chi(m_{1}) \tau(\bar{\chi})
\end{equation}
where $$\tau(\bar{\chi}):=\sum_{m=1}^{q_{1}} \bar{\chi}(m)e(m/q_{1}).$$
Note that 
$\tau(\bar{\chi})=\mu(q_{1})$ when $\chi$ is a principal character $\textrm{mod} \; q_{1},$ $|\tau(\bar{\chi})|\leq \sqrt{q_{1}}$ otherwise. 
Using \eqref{rationalidentity}, we have  
\begin{equation}\begin{split}\sum_{m=1}^{n} \alpha(m) e(am/q)&= \sum_{q=q_{0}q_{1}} \frac{1}{\phi(q_{1})}  \sum_{\chi \; \textrm{mod} \; q_{1}} \tau(\bar{\chi}) \chi(a) \sum_{m_{1} \leq n/q_{0}} \alpha(q_{0}m_{1})\chi(m_{1})
\\&  =\sum_{q=q_{0}q_{1}} \frac{\mu(q_{1})}{\phi(q_{1})} \sum_{m_{1} \leq n/q_{0} \atop (m_{1},q_{1})=1} \alpha(q_{0}m_{1})
\\& \;\; + O\left(\sum_{q=q_{0}q_{1}}q_{1}^{1/2+\epsilon} \left| \sum_{m_{1} \leq n/q_{0}} \alpha(q_{0}m_{1})\chi(m_{1})\right|\right).
\end{split}\end{equation}
\end{proof}
\noindent Next we estimate the main term in Lemma \ref{Lemma61}.
\begin{Lemma}\label{Lemma62} Assume the conditions in Proposition \ref{generalprop}, and let $n/2q_{0}$ be sufficiently large. Then
$$\sum_{n/2q_{0} \leq m_{1} \leq n/q_{0} \atop (m_{1},q_{1})=1} \alpha(q_{0}m_{1}) = w'_{q_{0},q_{1}}\textrm{Res}_{s=1} \frac{\mathcal{A}(s) (n/2q_{0})^{s}}{s} + O\left((n/2q_{0})^{4/5}\right)$$
where $$w'_{q_{0},q_{1}}= \prod_{p|q} 
\left(1+\frac{\alpha(p)}{p}+\frac{\alpha(p^{2})}{p^{2}}+...\right)^{-1}\prod_{p|q_{0} \atop p^{l}|q_{0}, p^{l+1} \nmid q_{0}}\left( \alpha(p^{l})+\frac{\alpha(p^{l+1})}{p}+\frac{\alpha(p^{l+2})}{p^2}+...\right). $$ 

\end{Lemma}
\begin{proof} The proof follows from 
the proof of \cite[Lemma 3.4]{Kim2022s}.
Let $0<Y<\frac{n}{10q_{0}}.$
Let $\eta$ be a fixed smooth function with a compact support $[n/2q_{0} -Y,n/q_{0}+Y]$ such that $\eta(\frac{m}{n/2q_{0}})=1$ for $ n/2q_{0} \leq m\leq n/q_{0}$, $\eta(x) \in [0,1]$. By the condition $\textit{(1)}$, we have

$$ \sum_{n/2q_{0} \leq m_{1} \leq n/q_{0} \atop (m_{1},q_{1})=1}\alpha(q_{0}m_{1})= \sum_{m_{1}=1 \atop (m_{1},q_{1})=1}^{\infty}\alpha(q_{0}m_{1})\eta(\frac{m_{1}}{n/2q_{0}}) + O_{\epsilon}\left( Yn^{\epsilon}q_{0}^{\epsilon}\right).$$
Let $$\tilde{\eta}(s):= \int_{0}^{\infty} \eta(x)x^{s-1}dx.$$
By the Mellin inversion formula, we get
$$\sum_{n/2q_{0} \leq m_{1} \leq n/q_{0} \atop (m_{1},q_{1})=1} \alpha(q_{0}m_{1}) =\frac{1}{2\pi i} \int_{2-i\infty}^{2+i\infty} E(s) \tilde{\eta}(s)(n/2q_{0})^{s} ds+O_{\epsilon}(Yn^{\epsilon}q_{0}^{\epsilon})$$
where
$$E(s):= \sum_{m=1, \atop (m,q_{1})=1}^{\infty} \alpha(q_{0}m)m^{-s}= \mathcal{A}(s) \prod_{p|q} 
\left(1+\frac{\alpha(p)}{p^{s}}+...\right)^{-1}\prod_{p|q_{0} \atop p^{l}|q_{0}, p^{l+1} \nmid q_{0}}\left( \alpha(p^{l})+\frac{\alpha(p^{l+1})}{p^{s}}+...\right)  $$
Let $\nu_{q_{0},q_{1}}= \textrm{Res}_{s=1}E(s) \times \tilde{\eta}(1)$. It is easy to check that $\tilde{\eta}(s) \ll_{l} (\frac{n}{2q_{0}Y})^{l-1} |s|^{-l}$ for any $l \in \mathbb{N}.$
By shifting the contour to the line $\Re(s)=\frac{1}{2},$ we have
$$\sum_{n/2q_{0} \leq m_{1} \leq n/q_{0}, \atop(m_{1},q_{1})=1} \alpha(q_{0}m_{1})=\nu_{q_{0},q_{1}} n/2q_{0}+ \int_{\frac{1}{2}-i\infty}^{\frac{1}{2}+i\infty} E(s)\tilde{\eta}(s)(n/2q_{0})^{s}ds + O_{\epsilon}(Yn^{\epsilon}q^{\epsilon}).$$
Note that for $\sigma \in [\frac{1}{2},1],$
\begin{equation}\begin{split}\prod_{p|q} &
\left(1+\frac{\alpha(p)}{p^{\sigma+it}}+...\right)^{-1}\prod_{p|q_{0} \atop p^{l}|q_{0}, p^{l+1} \nmid q_{0}}\left( \alpha(p^{l})+\frac{\alpha(p^{l+1})}{p^{\sigma+it}}+...\right)
\\&\ll \prod_{p|q} 
\left(1-\frac{2^{k+1}}{p^{\sigma}}-...\right)^{-1}\prod_{p|q_{0} \atop p^{l}|q_{0}, p^{l+1} \nmid q_{0}}\left( (l+1)^{k}+\frac{(l+2)^{k}}{p^{\sigma}}+...\right)
\\&\ll (\log_{2}(q+1))^{k+1} \prod_{p|q}\left( 1+\frac{2^{k}}{p^{\sigma}}+\frac{3^{k}}{p^{2\sigma}} +...\right)
\\&\ll  (\log_{2}(q+1))^{k+1} \prod_{p|q \atop p>2^{2k}}\left(1-\frac{1}{p^{\sigma}/2^{k}}\right)^{-1}
\\&\ll (\log_{2} (q+1))^{k+2}
\end{split}\end{equation}
\noindent (using the fact that $p^{l}|q \Rightarrow l \leq \log_{2} (q+1)$).
By the condition $\textit{(4)}$, $E(s) \ll_{\epsilon} (\log (q+1))^{k+2} |s|^{b_{1}}$  on $\Re(s)=\frac{1}{2}.$ Therefore, by  $\tilde{\eta}(s) \ll (\frac{n}{2q_{0}Y})^{b_{1}+1} |s|^{-b_{1}-2},$
\begin{equation}\begin{split}\int_{\frac{1}{2}-i\infty}^{\frac{1}{2}+i\infty} E(s)\tilde{\eta}(s)(n/2q_{0})^{s}ds
&\ll_{\epsilon} \Big|\int_{\frac{1}{2}-i\frac{(n/2q_{0})^{1+\epsilon}}{Y}}^{\frac{1}{2}+i\frac{(n/2q_{0})^{1+\epsilon}}{Y}} E(s)\tilde{\eta}(s)(n/2q_{0})^{s}ds\Big| \\&\;\;\;+ (\log (q+1))^{k+2}(n/2q_{0}Y)^{b_{1}}(n/2q_{0})^{1/2} \\ &\ll_{\epsilon}(n/2q_{0})^{\epsilon} (\log (q+1))^{k+2} (n/(2q_{0}Y))^{b_{1}}(n/2q_{0})^{1/2}.\end{split} \nonumber \end{equation} 
When $Y=(n/2q_{0})^{1-\frac{1}{4b_{1}}},$ the crude bound of the above inequality is  $O\left((n/2q_{0})^{4/5}\right).$
By the conditions on $\eta$, we have  $\tilde{\eta}(1)=1+O\left((\frac{n}{2q_{0}})^{-\frac{1}{4}}\right).$ Therefore, 
$$v_{q_{0},q_{1}}n/2q_{0} = w'_{q_{0},q_{1}} \textrm{Res}_{s=1} \frac{\mathcal{A}(s)(n/2q_{0})^{s}}{s}+O \left( (n/2q_{0})^{4/5}\right).$$
\end{proof}
\begin{Rem} The bound for the error term in Lemma \ref{Lemma62} can be improved, but it is sufficient for our purposes.
\end{Rem}
\noindent Using a similar argument as in the proof of Proposition \ref{33}, it is easy to prove the $\alpha(m)$ version of Proposition \ref{33}. Since the proof of the following lemma is very similar to the proof of Lemma \ref{Lemma62}, we will skip it.  The only difference is that we apply the condition $\textit{(3)}$ instead of $\textit{(2)}.$ 
\begin{Lemma}\label{Lemma63}  Assume the conditions in Proposition \ref{generalprop}, and let $n/2q_{0}$ be sufficiently large. Then, for any non-principal Dirichlet character $\chi$ mod $q_{1},$
$$\sum_{n/2q_{0} \leq m_{1} \leq n/q_{0}} \alpha(q_{0}m_{1})\chi(m_{1}) \ll \left((n/2q_{0})^{4/5}q_{1}^{b_{2}}\right).$$
\end{Lemma}
\subsection{Proof of Proposition \ref{generalprop}} Let us first prove the conclusion $(a)$.
By using dyadic summation with Lemma \ref{Lemma62}, for sufficiently large $n/q,$ we have
$$
\sum_{q=q_{0}q_{1}} \frac{\mu(q_{1})}{\phi(q_{1})} \sum_{m_{1} \leq n/q_{0} \atop (m_{1},q_{1})=1} \alpha(q_{0}m_{1})= \sum_{q=q_{0}q_{1}} \frac{\mu(q_{1})w'_{q_{0},q_{1}}}{\phi(q_{1})} \textrm{Res}_{s=1} \frac{\mathcal{A}(s)(n/q_{0})^{s}}{s}+O_{\epsilon} \left( n^{4/5}q^{-1+\epsilon}\right).$$
Since $\mathcal{A}(s)$ has a pole of order $\varkappa$ at $s=1$, the residue 
 $$\textrm{Res}_{s=1} \frac{\mathcal{A}(s)(n/2q_{0})^{s}}{s}$$ is a degree $\varkappa-1$ polynomial in $\log (n/q_{0}).$
 Therefore, 
$$ \sum_{q=q_{0}q_{1}} \frac{\mu(q_{1})}{\phi(q_{1})} \sum_{m_{1} \leq n/q_{0} \atop (m_{1},q_{1})=1} \alpha(q_{0}m_{1}) = n\sum_{i=1}^{\varkappa} \mathcal{C}_{q,i} (\log n)^{\varkappa-i} + O_{\epsilon} \left(n^{\frac{4}{5}}q^{-1+\epsilon}\right)$$ for some constants $\mathcal{C}_{q,i} \ll_{\epsilon} q^{-1+\epsilon}.$
Using Lemma \ref{Lemma63}, we have the crude bound
$$\sum_{q=q_{0}q_{1}}q_{1}^{\frac{1}{2}+\epsilon} \left| \sum_{m_{1} \leq n/q_{0}} \alpha(q_{0}m_{1})\chi(m_{1})\right| \ll n^{\frac{4}{5}}q^{b_{2}+1}. 
$$ 
Now, we are ready to prove $(b).$
Let $\beta= x-a/q$ such that $|\beta|\leq \frac{1}{Q_{n}}.$
Using \cite[Lemma 2.1]{MRT1}, we see that 
$$
\sum_{1 \leq m \leq n} \alpha(m) e\left(am/q\right) e(\beta m) 
=\int_0^{n} \frac{d}{dx}\left(xF_{q}(x)\right) e(\beta x) d x+O_\epsilon\left((|\beta|+1/n) n^{9/5} q^{b_{2}+1}\right).
$$
Since $$G_{q}(x):=\frac{d}{dx}\left(xF_{q}(x)\right)= \sum_{i=1}^{\varkappa} \mathcal{C}'_{q,i}(\log x)^{\varkappa-i}$$ for some constants $\mathcal{C}'_{q,i} \ll_{\epsilon} q^{-1+\epsilon},$ 
using a similar argument as in the proof of Proposition \ref{Prop42}, we have  
\begin{equation}\begin{split}\int_0^{n} \frac{d}{dx}\left(xF_{q}(x)\right) e(\beta x) d x &= G_{q}(n)\sum_{m=1}^{n} e(\beta m)  \\&\;\;\;+O_\epsilon\left(q^{-1+\epsilon}(\log n)^{\varkappa-1} / 2+q^{-1+\epsilon} n(\log n)^{\varkappa-1}\left(\frac{1}{\log n}+1 / Q_n\right)\right).\end{split}\end{equation}
Therefore, 
\begin{equation}
\sum_{m=1}^{n} \alpha(m) e((x-a/q)m) = G_{q}(n) \sum_{m=1}^{n} e\left(\beta m\right) +O \left(n^{4/5}q^{b_{2}+1}\left(1+nQ_{n}^{-1}+n^{1/5}q^{-b_{2}-2+\epsilon}(\log n)^{\varkappa-2}\right)\right). 
\end{equation}
Since $Q_{n}= n (\log n)^{-2(1-\epsilon)},$ the proof of $(b)$ is completed. 
Using a similar argument as in the proof of Lemma \ref{3lemma}, the proof of $(c)$ is completed.
\begin{Rem}Using a similar argument as in Theorems \ref{hecketheorem} and \ref{mainthm},
 for any ergodic dynamical system $(X, \mathcal{A}, \nu, \tau)$ and any $f$ in $L^p(\nu)$ where $p \in (3/2, \infty),$
$$\lim_{n\rightarrow \infty} \frac{1}{\sum_{k=1}^{n} \alpha(k)} \sum_{k=1}^{n} \alpha(k)f\left(\tau^{k}(x)\right)$$ converges $\nu$- almost everywhere. For this case, one only needs to replace $ \psi_{n,q}(x), \omega_{n,q}(x), T_{n}(x)$ with 
$$ \frac{G_{q}(n)}{\sum_{m=1}^{n} \alpha(m)} \sum_{m=1}^{n} e(mx), \;\;\frac{G_{q}(n)}{\sum_{m=1}^{n} \alpha(m)} \sum_{m=1}^{n} \delta_{m}, \;\; \frac{\sum_{m=1}^{n} \alpha(m) e(mx)}{\sum_{m=1}^{n} \alpha(m)}
,$$ respectively. Since 
$$ \frac{G_{q}(n)}{\sum_{m=1}^{n} \alpha(m)} \sum_{m=1}^{n} \delta_{m}$$ is a scalar multiple of the Ces\'aro kernel, the required maximal inequality and the oscillation inequality can be derived using the same arguments as in the proof of \eqref{maximal} and  Proposition \ref{34}. 
\end{Rem}

\subsection*{Acknowledgements} The author would like to thank his advisor Professor Xiaoqing Li, for her constant support. Additionally, we extend our appreciation to the anonymous referee for providing us with valuable suggestions. 
\bibliographystyle{plain}   
\bibliography{over}  

\end{document}